\documentclass[12pt,reqno]{amsart}
\usepackage{amscd,amssymb,amsmath,multicol,float,scalefnt,cancel,array,stmaryrd,graphicx,tikz}
\restylefloat{table}
\newtheorem{thm}[equation]{Theorem}
\numberwithin{equation}{section}
\newtheorem{cor}[equation]{Corollary}

\newtheorem{rmk}[equation]{Remark}

\newtheorem{lem}[equation]{Lemma}

\newtheorem{prop}[equation]{Proposition}

\newtheorem{fig}[equation]{Figure}

\begin{document}
\raggedbottom \voffset=-.7truein \hoffset=0truein \vsize=8truein
\hsize=6truein \textheight=8truein \textwidth=6truein
\baselineskip=18truept

\def\mapright#1{\ \smash{\mathop{\longrightarrow}\limits^{#1}}\ }
\def\mapleft#1{\smash{\mathop{\longleftarrow}\limits^{#1}}}
\def\mapup#1{\Big\uparrow\rlap{$\vcenter {\hbox {$#1$}}$}}
\def\mapdown#1{\Big\downarrow\rlap{$\vcenter {\hbox {$\ssize{#1}$}}$}}
\def\mapne#1{\nearrow\rlap{$\vcenter {\hbox {$#1$}}$}}
\def\mapse#1{\searrow\rlap{$\vcenter {\hbox {$\ssize{#1}$}}$}}
\def\mapr#1{\smash{\mathop{\rightarrow}\limits^{#1}}}
\def\ss{\smallskip}
\def\s{\sigma}
\def\l{\lambda}
\def\vp{v_1^{-1}\pi}
\def\at{{\widetilde\alpha}}
\def\sm{\wedge}
\def\la{\langle}
\def\ra{\rangle}
\def\ev{\text{ev}}
\def\od{\text{od}}
\def\on{\operatorname}
\def\ol#1{\overline{#1}{}}
\def\spin{\on{Spin}}
\def\cat{\on{cat}}
\def\lbar{\ell}
\def\qed{\quad\rule{8pt}{8pt}\bigskip}
\def\ssize{\scriptstyle}
\def\a{\alpha}
\def\bz{{\Bbb Z}}
\def\Rhat{\hat{R}}
\def\im{\on{im}}
\def\ct{\widetilde{C}}
\def\ext{\on{Ext}}
\def\sq{\on{Sq}}
\def\eps{\epsilon}
\def\ar#1{\stackrel {#1}{\rightarrow}}
\def\br{{\bold R}}
\def\bC{{\bold C}}
\def\bA{{\bold A}}
\def\bB{{\bold B}}
\def\bD{{\bold D}}
\def\bh{{\bold H}}
\def\bQ{{\bold Q}}
\def\bP{{\bold P}}
\def\bx{{\bold x}}
\def\bo{{\bold{bo}}}
\def\si{\sigma}
\def\Vbar{{\overline V}}
\def\dbar{{\overline d}}
\def\wbar{{\overline w}}
\def\Sum{\sum}
\def\tfrac{\textstyle\frac}
\def\tb{\textstyle\binom}
\def\Si{\Sigma}
\def\w{\wedge}
\def\equ{\begin{equation}}
\def\b{\beta}
\def\G{\Gamma}
\def\L{\Lambda}
\def\g{\gamma}
\def\k{\kappa}
\def\psit{\widetilde{\Psi}}
\def\tht{\widetilde{\Theta}}
\def\psiu{{\underline{\Psi}}}
\def\thu{{\underline{\Theta}}}
\def\aee{A_{\text{ee}}}
\def\aeo{A_{\text{eo}}}
\def\aoo{A_{\text{oo}}}
\def\aoe{A_{\text{oe}}}
\def\vbar{{\overline v}}
\def\endeq{\end{equation}}
\def\sn{S^{2n+1}}
\def\zp{\bold Z_p}
\def\cR{{\mathcal R}}
\def\P{{\mathcal P}}
\def\cQ{{\mathcal Q}}
\def\cj{{\cal J}}
\def\zt{{\bold Z}_2}
\def\bs{{\bold s}}
\def\bof{{\bold f}}
\def\bq{{\bold Q}}
\def\be{{\bold e}}
\def\Hom{\on{Hom}}
\def\ker{\on{ker}}
\def\kot{\widetilde{KO}}
\def\coker{\on{coker}}
\def\da{\downarrow}
\def\colim{\operatornamewithlimits{colim}}
\def\zphat{\bz_2^\wedge}
\def\io{\iota}
\def\om{\omega}
\def\Prod{\prod}
\def\e{{\cal E}}
\def\zlt{\Z_{(2)}}
\def\exp{\on{exp}}
\def\abar{{\overline a}}
\def\xbar{{\overline x}}
\def\ybar{{\overline y}}
\def\zbar{{\overline z}}
\def\mbar{{\overline m}}
\def\nbar{{\overline n}}
\def\sbar{{\overline s}}
\def\kbar{{\overline k}}
\def\bbar{{\overline b}}
\def\et{{\widetilde E}}
\def\ni{\noindent}
\def\tsum{\textstyle \sum}
\def\coef{\on{coef}}
\def\den{\on{den}}
\def\lcm{\on{l.c.m.}}
\def\vi{v_1^{-1}}
\def\ot{\otimes}
\def\psibar{{\overline\psi}}
\def\thbar{{\overline\theta}}
\def\mhat{{\hat m}}
\def\exc{\on{exc}}
\def\ms{\medskip}
\def\ehat{{\hat e}}
\def\etao{{\eta_{\text{od}}}}
\def\etae{{\eta_{\text{ev}}}}
\def\dirlim{\operatornamewithlimits{dirlim}}
\def\gt{\widetilde{L}}
\def\lt{\widetilde{\lambda}}
\def\st{\widetilde{s}}
\def\ft{\widetilde{f}}
\def\sgd{\on{sgd}}
\def\lfl{\lfloor}
\def\rfl{\rfloor}
\def\ord{\on{ord}}
\def\gd{{\on{gd}}}
\def\rk{{{\on{rk}}_2}}
\def\nbar{{\overline{n}}}
\def\MC{\on{MC}}
\def\lg{{\on{lg}}}
\def\cH{\mathcal{H}}
\def\cS{\mathcal{S}}
\def\cP{\mathcal{P}}
\def\N{{\Bbb N}}
\def\Z{{\Bbb Z}}
\def\Q{{\Bbb Q}}
\def\R{{\Bbb R}}
\def\C{{\Bbb C}}

\def\mo{\on{mod}}
\def\xt{\times}
\def\notimm{\not\subseteq}
\def\Remark{\noindent{\it  Remark}}
\def\kut{\widetilde{KU}}

\def\*#1{\mathbf{#1}}
\def\0{$\*0$}
\def\1{$\*1$}
\def\22{$(\*2,\*2)$}
\def\33{$(\*3,\*3)$}
\def\ss{\smallskip}
\def\ssum{\sum\limits}
\def\dsum{\displaystyle\sum}
\def\la{\langle}
\def\ra{\rangle}
\def\on{\operatorname}
\def\proj{\on{proj}}
\def\od{\text{od}}
\def\ev{\text{ev}}
\def\o{\on{o}}
\def\U{\on{U}}
\def\lg{\on{lg}}
\def\a{\alpha}
\def\bz{{\Bbb Z}}
\def\eps{\varepsilon}
\def\bc{{\bold C}}
\def\bN{{\bold N}}
\def\nut{\widetilde{\nu}}
\def\tfrac{\textstyle\frac}
\def\b{\beta}
\def\G{\Gamma}
\def\g{\gamma}
\def\zt{{\Bbb Z}_2}
\def\zth{{\bold Z}_2^\wedge}
\def\bs{{\bold s}}
\def\bx{{\bold x}}
\def\bof{{\bold f}}
\def\bq{{\bold Q}}
\def\be{{\bold e}}
\def\lline{\rule{.6in}{.6pt}}
\def\xb{{\overline x}}
\def\xbar{{\overline x}}
\def\ybar{{\overline y}}
\def\zbar{{\overline z}}
\def\ebar{{\overline \be}}
\def\nbar{{\overline n}}
\def\ubar{{\overline u}}
\def\bbar{{\overline b}}
\def\et{{\widetilde e}}
\def\lf{\lfloor}
\def\rf{\rfloor}
\def\ni{\noindent}
\def\ms{\medskip}
\def\hh{{\widehat h}}
\def\kk{{\widehat k}}
\def\Dhat{{\widehat D}}
\def\abar{{\overline{a}}}
\def\minp{\min\nolimits'}
\def\mul{\on{mul}}
\def\N{{\Bbb N}}
\def\Z{{\Bbb Z}}
\def\S{\Sigma}
\def\Q{{\Bbb Q}}
\def\R{{\Bbb R}}
\def\C{{\Bbb C}}
\def\notint{\cancel\cap}
\def\cS{\mathcal S}
\def\cR{\mathcal R}
\def\el{\ell}
\def\TC{\on{TC}}
\def\GC{\on{GC}}
\def\wgt{\on{wgt}}
\def\wpt{\widetilde{p_2}}
\def\dstyle{\displaystyle}
\def\Om{\Omega}
\def\ds{\dstyle}
\def\tz{tikzpicture}
\def\zcl{\on{zcl}}
\def\Vb#1{{\overline{V_{#1}}}}
\title
{Geodesics in the configuration spaces of two points in $\R^n$}
\author{Donald M. Davis}
\address{Department of Mathematics, Lehigh University\\Bethlehem, PA 18015, USA}
\email{dmd1@lehigh.edu}
\date{February 19, 2020}


\maketitle
\begin{abstract} We determine explicit formulas for geodesics (in the Euclidean metric) in the configuration space of ordered pairs $(x,x')$ of points in $\R^n$ which satisfy $d(x,x')\ge \eps$. We interpret this as two or three (depending on the parity of $n$) geodesic motion-planning rules for this configuration space. In the associated unordered configuration space, we need not prescribe that the points stay apart by $\eps$. For this space, with a Euclidean-related metric, we show that geodesic motion-planning rules correspond to ordinary motion-planning rules on $RP^{n-1}$.
 \end{abstract}
\section{Results}\label{intro}
Recently David Recio-Mitter (\cite{RM}) introduced the notion of geodesic complexity, which is an analogue of Farber's topological complexity (\cite{Far}), but requires that paths be minimal geodesics. This is a useful requirement for efficient motion-planning algorithms. In \cite{RM} and \cite{DR}, the geodesic complexity of several spaces was determined.

Configuration spaces are of central importance in topological robotics, since they model the situation of several robots moving throughout a region. In this paper, we first consider the case of two distinguished points (or balls) moving in $\R^n$. We obtain explicit formulas for the geodesics and optimal geodesic motion-planning rules. We also consider two indistinguishable points moving in $\R^n$, and show that geodesic motion-planning rules for these correspond to ordinary motion-planning rules in real projective space $RP^{n-1}$.

Let $F(\R^n,2)$ denote the ordered configuration space of two distinct points in $\R^n$. It is a subspace of $\R^{2n}$ and is given the Euclidean metric. This space is not geodesically complete. For example, there is no geodesic from $((1,\overline 0),(\overline 0,1))$ to $((-1,\overline 0),(\overline 0,-1))$ since the linear path $\s(t)=((1-2t,\overline 0),(\overline 0,1-2t))$ has $\sigma(\frac12)\not\in F(\R^n,2)$, but there are paths in $F(\R^n,2)$ between these points arbitrarily close to $\s$. By ``geodesic,'' we will always mean ``minimal geodesic.''

For a positive number $\eps$, we consider the subspace of $F(\R^n,2)$ consisting of points $(x,x')$ for which $d(x,x')\ge\eps$. By scaling, we may assume $\eps=2$, and define
$$F_0(\R^n,2)=\{(x,x')\in F(\R^n,2):\ d(x,x')\ge2\}.$$
This can be viewed as the space of ordered pairs of disjoint open unit balls in $\R^n$. Note that $F_0(\R^n,2)$ is a manifold with boundary $\partial F_0$ consisting of points of the form $(x-u,x+u)$ with $\|u\|=1$.
In the following theorem, we  give explicit formulas for geodesics in $F_0(\R^n,2)$ between any two points.

\begin{thm}\label{mainthm} Let $P=(a,a')$ and $Q=(b,b')$ be points of $F_0(\R^n,2)$. Let
$$h=(a'-a)/2,\quad k=(b'-b)/2,\quad A=(a'+a)/2, \quad B=(b'+b)/2.$$
Let $\delta=\min\{d(tb+(1-t)a,tb'+(1-t)a'):\ 0\le t\le1\}$, the minimal distance between the two components of the linear path between $P$ and $Q$.
\begin{itemize}
\item[a.] If $\delta\ge2$, the linear path from $P$ to $Q$ is the unique geodesic between $P$ and $Q$ in $F_0(\R^n,2)$.
\item[b.] If $0<\delta\le2$, there is a unique geodesic in $F_0(\R^n,2)$ from $P$ to $Q$. It is the path composition $\ell_1\sigma\ell_2$, where $\ell_1$ is the linear path from $P$ to $C_0=(x-u,x+u)$, $\sigma$ the geodesic in $\partial F_0$ from $C_0$ to $C_1=(y-v,y+v)$, described in Proposition \ref{bdrylem}, and $\ell_2$ the linear path from $C_1$ to $Q$.
    Here $u$ and $v$ are  unique unit vectors in $\R^n$ satisfying
\begin{equation}\label{=1}h\cdot u=1\text{ and }k\cdot v=1\end{equation}
with minimal $\|u-v\|$. Let $\b$ be the angle between this $u$ and $v$ with $0\le\b<\pi$. Then
 \begin{equation}\label{xy}x=\frac{\b A+S_0B+S_1A}{\b+S_0+S_1},\quad y=\frac{\b B+S_0B+S_1A}{\b+S_0+S_1},
 \end{equation}
 where $S_0=\sqrt{\|h\|^2-1}$ and $S_1=\sqrt{\|k\|^2-1}$. If $\b=0$, then $C_0=C_1$, $\delta=2$, and $\ell_1\ell_2$ is the linear path from $P$ to $Q$. When $\delta=2$, the linear path in (a) can also be obtained by the method of (b).
\item[c.] If $\delta=0$, then $h$ and $k$ are parallel in opposite directions, and conversely. In this case,
the unit vector solutions $u,v$
 of (\ref{=1}) with minimal $\|u-v\|$ are
 $$u=\frac h{\|h\|^2}+\frac{S_0}{\|h\|}w,\quad v=\frac k{\|k\|^2}+\frac{S_1}{\|k\|}w,$$
 where $w$ ranges over the set of all points satisfying $h\cdot w=0$ and $w\cdot w=1$.
 The geodesics from $P$ to $Q$ are paths as described in (b) for each of these pairs $u,v$, using (\ref{xy}) and Corollary \ref{cor1}.
 \end{itemize}
 \end{thm}
 \noindent We will refer to these as type (a), (b), or (c) paths or Situations. In Proposition \ref{minl} and Corollary \ref{cor1} we give explicit formulas for $u$, $v$, and $\b$  in terms of $h$ and $k$.

 The geodesics in $\partial F_0$ to which we just referred are described in the following result.
 \begin{prop}\label{bdrylem} Let $u$ and $v$ be unit vectors in $\R^n$, and $\a$ be the angle from $u$ to $v$ with $0\le\a<\pi$. For $0\le t\le 1$, let
$$u(t)=\frac{\sin((1-t)\a)u+\sin(t\a)v}{\sin\a}$$
if $\a>0$. If $\a=0$, then $u(t)=u=v$ for all $t$.
For $x,y\in\R^n$, the unique geodesic in $\partial F_0$ from $(x-u,x+u)$ to $(y-v,y+v)$ is the curve
$$\sigma(t)=((1-t)x+ty-u(t),(1-t)x+ty+u(t)).$$
Its length is \begin{equation}\label{len}\sqrt{2(\|x-y\|^2+\a^2)}.\end{equation}
 \end{prop}

 Recall that the geodesic complexity $\GC(X)$ is the smallest $k$ such that $X\times X$ can be partitioned into ENRs $E_0,\ldots,E_k$ such that on each $E_i$ there is a continuous map $s_i$ from $E_i$ to the free path space $PX$ such that $s_i(x_0,x_1)$ is a geodesic from $x_0$ to $x_1$. (\cite{RM}) The topological complexity $\TC(X)$ is defined similarly without requiring that the paths be geodesics.(\cite{Far}) Our second result is the determination of $\GC(F_0(\R^n,2))$.
 \begin{thm}\label{GCthm} For $n\ge2$,
 $$\GC(F_0(\R^n,2))=\TC(F_0(\R^n,2))=\TC(S^{n-1})=\begin{cases}1&n\text{ even}\\ 2&n\text{  odd.}\end{cases}$$
 \end{thm}

The unordered configuration space $C(\R^2,2)$ is the quotient of $F(\R^n,2)$ by the involution which reverses the order of the two points. Points of $C(\R^n,2)$ are sets $\{a,a'\}$ with $a, a'\in\R^n$. Surprisingly, $C(\R^n,2)$ is, in some sense, easier for these considerations than $F(\R^n,2)$.
\begin{thm}\label{Cthm} With $d$ denoting the Euclidean metric in $\R^n\times \R^n$, defining
$$d_U(\{ a,a'\},\{b,b'\})=\min(d((a,a'),(b,b')),\ d((a,a'),(b',b)))$$
gives a metric on $C(\R^n,2)$ which has linear geodesics between any two points.\end{thm}
So, we need not bother with the intricacies for geodesics in $F(\R^n,2)$ caused by the need to keep points at least a certain distance apart.
The space $C(\R^n,2)$ has the homotopy type of $RP^{n-1}$, and so the following result, which we prove  in Section \ref{unorsec}, may not be surprising. The proof will show that the geodesics in $C(\R^n,2)$ are obtained from not-necessarily-geodesic paths in $RP^{n-1}$.
\begin{thm}\label{GCCthm} For $n\ge2$, $\GC(C(\R^n,2))=\TC(RP^{n-1})$.\end{thm}
\noindent By \cite{FTY}, $\TC(RP^n)$ equals the immersion dimension of $RP^n$ unless $n=1$, 3, or 7, but this does not enter into our proof.

In Section \ref{sec4}, we consider a different metric on $F(\R^n,2)$ in which it is geodesically complete, and discuss geodesics in that metric.
\section{Proof of Theorem \ref{mainthm}}\label{pfsec}
The following proof of Proposition \ref{bdrylem} benefited from ideas of David L.~Johnson.
\begin{proof}[Proof of Proposition \ref{bdrylem}]
With $u\cdot v=\cos\a$ and $u\cdot u=1=v\cdot v$, $u(t)=c_0u+c_1v$ is obtained by solving $u\cdot(c_0u+c_1v)=\cos(t\a)$ and $(c_0 u+c_1 v)\cdot(c_0u+c_1v)=1$. We have
\begin{eqnarray*}\sigma'(t)&=&(-x+y-\tfrac{\a}{\sin\a}(-\cos((1-t)\a)u+\cos(t\a)v),\\
&&-x+y+\tfrac{\a}{\sin\a}(-\cos((1-t)\a)u+\cos(t\a)v)).\end{eqnarray*}
Expanding $\cos(\a-t\a)$ yields
$\|-\cos((1-t)\a)u+\cos(t\a)v\|^2=\sin^2\a,$ and hence $\|\sigma'(t)\|^2=2(\|x-y\|^2+\a^2)$, which implies the claim about the length of the curve.

A constant-speed curve $\sigma$ with $\sigma''$ orthogonal to the surface is a geodesic. We have
$$\sigma''(t)=\tfrac{\a^2}{\sin\a}(\sin((1-t)\a)u+\sin(t\a)v,-(\sin((1-t)\a)u+\sin(t\a)v)).$$
The surface $\partial F_0$ is parametrized by
$X(x,u)=(x-u,x+u)$ with $x\in\R^n$ and $u\in S^{n-1}$. Then $\sigma''(t)$ is orthogonal to the $x$-directions and is orthogonal to the spherical parameter $u$ since it is a multiple of the radius vector at each point.

Since $\sigma'(0)=(-x+y-\frac\a{\sin\a}(-\cos(\a) u+v),-x+y+\frac\a{\sin\a}(-\cos(\a) u+v))$, every tangent direction from the initial point $(x-u,x+u)$ is obtained for one of our geodesics, showing that they are unique.
\end{proof}

The following lemma will be very important to our analysis.
\begin{lem}\label{anglelem} Let  $(a,a')\in F_0(\R^n,2)$, and let $h=(a'-a)/2\in\R^n$. If $u$ is a unit vector in $\R^n$ and $x\in\R^n$, the segment between $(a,a')$ and $(x-u,x+u)$ lies in $F_0(\R^n,2)$ iff $h\cdot u\ge1$.
\end{lem}
\begin{proof} We require that for $t\in[0,1]$
$$d(ta+(1-t)(x-u),ta'+(1-t)(x+u))\ge2.$$
Halving and squaring, this becomes
\begin{eqnarray*}1&\le&\|th+(1-t)u\|^2\\
&=&t^2\|h\|^2+2t(1-t)h\cdot u+(1-t)^2.\end{eqnarray*}
By assumption, $\|h\|^2\ge1$, so this quadratic function $f(t)$ satisfies $f(0)=1$ and $f(1)\ge1$. It is $\ge1$ for all $t\ge0$ iff $f'(0)\ge0$. Since $f'(0)=-2+2h\cdot u$, the result follows.
\end{proof}

 Because of Lemma \ref{anglelem}, paths of the form $\ell_1\sigma\ell_2$ in Theorem \ref{mainthm}(b) exist as long as $h\cdot u\ge1$ and $k\cdot v\ge1$ for unit vectors $u$ and $v$, for any $x$ and $y$. The proof of Theorem \ref{mainthm} will show that minimal length of such paths is achieved when $h\cdot u=1=k\cdot v$ and $\|u-v\|$ is minimized. The following result relates intersections of the hyperplanes $h\cdot u=1$ and $k\cdot u=1$ inside and on the unit sphere to the value of $\delta$ in Theorem \ref{mainthm}, which in turn determines the types of geodesics.

\begin{prop}\label{HKprop} Let $h$, $k$, and $\delta$ be as in Theorem \ref{mainthm}. Let $H=\|h\|^2\ge1$, $K=\|k\|^2\ge1$, and $D=h\cdot k$.
\begin{itemize}
\item If $\min(H,K)\le D$, then $\delta^2=4\min(H,K)\ge4$, so $\delta\ge2$.
\item If $\min(H,K)\ge D$ and $h\ne k$, then $$\delta^2=4(HK-D^2)/(H+K-2D).$$ In this case, regarding solutions of $h\cdot u=1$ and $k\cdot u=1$, we have
    \begin{itemize}
   \item [i.] There exist solutions with $\|u\|<1$ (and more than one solution with $\|u\|=1$) iff $\delta>2$.
 \item[ii.] There exists a unique solution with $\|u\|=1$ iff $\delta=2$.
 \item[iii.] There exist no solutions with $\|u\|\le1$ iff $\delta<2$.
 \end{itemize}
 \end{itemize}
 \end{prop}
 \begin{proof}  Note that  $HK\ge D^2$ by Cauchy-Schwarz, and $H+K>2D$, since $\|h\|^2+\|k\|^2>2\|h\|\,\|k\|\cos\a$ when $h\ne k$.

Let $d^2(t)=\|2tk+2(1-t)h\|^2=4(t^2K+(1-t)^2H+2t(1-t)D)$. Then $\delta^2=\min(d^2(t):\ t\in [0,1])$.
The minimum of $d^2(t)$ over all $t\in\R$ occurs when $t=t_0:=\frac{H-D}{H+K-2D}$, and has value $4\frac{HK-D^2}{H+K-2D}$.
Note that $t_0\in[0,1]$ iff $\min(H,K)\ge D$. If $\min(H,K)\le D$, then $d^2(t)$ does not have a relative minimum for $0<t<1$ so its absolute minimum on $[0,1]$ occurs at an endpoint.

If $k$ is a scalar multiple of $h$, the result is easily verified, so we assume this is not the case, and have $HK-D^2>0$.
Now let $u=\a h+\b k+\ell$ with $\ell$ orthogonal to $h$ and $k$. If $n=2$, omit $\ell$.
The equations $h\cdot u=1$ and $k\cdot u=1$ become $\a H+\b D=1$ and $\a D+\b K=1$, whose solution is
$$\a=\frac{K-D}{HK-D^2},\qquad \b=\frac{H-D}{HK-D^2},$$
yielding
\begin{equation}\label{uu}u\cdot u=\frac{H+K-2D}{HK-D^2}+\ell\cdot\ell.\end{equation}

Thus when $k$ is not a scalar multiple of $h$, we have $\|u\|^2=\frac4{\delta^2}+\|\ell\|^2$. The conclusions follow. By the complementary nature of the cases, it suffices to show implication in one direction. For each hypothesis on $\delta$, the conclusion about $\|u\|$ for solutions is clear. In case (i), the solutions are obtained by varying $\ell$.
\end{proof}

By Proposition \ref{HKprop}, the next result applies exactly when $\delta\le2$.
 \begin{prop}\label{minl} Let $h$ and $k$ satisfy $\|h\|>1$ and $\|k\|>1$. Assume there does not exist $u$ with $h\cdot u=1=k\cdot u$ with $\|u\|<1$. The solutions of $h\cdot u=1=k\cdot v$ and $\|u\|=1=\|v\|$ with minimal $\|u-v\|$ are
 \begin{itemize}\item[i.] If $k$ is a scalar multiple of $h$, then
 $$u=\frac h{\|h\|^2}+\frac{S_0}{\|h\|}w,\quad v=\frac k{\|k\|^2}+\frac{S_1}{\|k\|}w,$$
 where $w$ ranges over the set of vectors satisfying $h\cdot w=0$ and $w\cdot w=1$. Here $S_0$ and $S_1$ are as in Theorem \ref{mainthm}.
 \item[ii.] If $k$ is not a scalar multiple of $h$, then there is a unique solution, using notation of Proposition \ref{HKprop},
 \begin{eqnarray*}u&=&\frac hH+\frac{S_0(k-\frac DHh)}{\sqrt{HK-D^2}}\\
 v&=&\frac kK+\frac{S_1(h-\frac DKk)}{\sqrt{HK-D^2}}.\end{eqnarray*}
     \end{itemize}
     \end{prop}
\begin{proof} (i.) The vectors $u$ and $v$ lie on the intersections with the unit sphere of parallel hyperplanes.
The vector $u$ can be written uniquely as $u=c_0h+c_1w$ satisfying $h\cdot u=1$, $u\cdot u=1$, $w\cdot w=1$, and $h\cdot w=0$. These equations yield the above formula for $u$, and similarly for $v$. The vector $v$ closest to $u$ will be the one with the same unit vector $w$, and for all vectors $w$, the values of $\|u-v\|$ are the same.

(ii.) Let $\ell=k-\frac DHh$, which is orthogonal to $h$. Then $v$ can be written uniquely as $ah+b\ell+m$, with $m$ orthogonal to $h$ and $\ell$. The point $v$ closest to the hyperplane $h\cdot u=1$ will be the one with the largest $a$. From $v\cdot k=1$, we deduce $1=aD+b(HK-D^2)/H$. Use this to eliminate $b$. Then $v\cdot v=1$ implies
$$1=a^2H+\biggl(\frac{H(1-aD)}{HK-D^2}\biggr)^2\biggl(\frac{HK-D^2}H\biggr)+M,$$
where $M=\|m\|^2$. This simplifies to
$$0=HK\,a^2-2Da+1-(1-M)(HK-D^2)/H,$$
which has solution
$$a=\frac{D\pm\sqrt{D^2+K((1-M)(HK-D^2)-H)}}{HK}.$$
The maximum of this occurs when $M=0$ and has $a=(D+\sqrt{(K-1)(HK-D^2)})/HK$. One easily finds $b$ now and obtains the formula for $v$. The formula for $u$ is obtained similarly.
\end{proof}

\begin{cor}\label{cor1} The angle $\b$ in Theorem \ref{mainthm}(b) satisfies
$$\cos(\b)=\frac{(S_0+S_1)\sqrt{HK-D^2}+(1-S_0S_1)D}{HK}.$$
The angle $\b$ in Theorem \ref{mainthm}(c) satisfies
$$\cos(\b)=\frac D{HK}+\frac{S_0S_1}{\sqrt{HK}}.$$
\end{cor}
\begin{proof} We compute $\cos(\b)=u\cdot v$ from Proposition \ref{minl}.\end{proof}
\begin{proof}[Proof of Theorem \ref{mainthm}]
The conclusion of part (a) is immediate. Next we reduce consideration of part (c) to that of part (b).

First note that $\delta=0$ iff $tb+(1-t)a=tb'+(1-t)a'$ for some $t\in[0,1]$ iff $b'-b$ and $a'-a$ are negative multiples of one another. Then the analysis of type-(b) paths which follows applies, with minor modifications which are discussed below, to all of the pairs $u,v$ obtained in Proposition \ref{minl}(i), listed again in Theorem \ref{mainthm}(c).

Geodesics in a manifold with boundary are path compositions of geodesics in the manifold and geodesics in the boundary.(e.g., \cite{Hu}.) In our case, this will consist of at most one geodesic in $\partial F_0$. [\![If it were $\ell_1\sigma_1\ell_2\sigma_2\ell_3$, then $\ell_2$ would be a line segment connecting two points of $\partial F_0$. Similarly to the proof of Lemma \ref{anglelem}, the line segment connecting two points $(x-u,x+u)$ and $(y-v,y+v)$ of $\partial F_0$ will lie outside $F_0(\R^n,2)$ unless the two points have $u=v$, in which case it is a line segment lying in $\partial F_0$. When path-multiplied by an angle-changing geodesic in $\partial F_0$, the result will not be a geodesic.]\!]
So we need just consider path compositions of the form $\ell_1\sigma\ell_2$.

If $\|h\|=1$, then $P\in\partial F_0$. By the argument just described, we may then choose $\ell_1$ to be the constant path. This is consistent with (\ref{xy}) since we would have $S_0=0$, $x=A$, and $u=h$. Similarly, if $\|k\|=1$, the path $\ell_2$ may be ignored. Thus we shall assume $\|h\|>1$ and $\|k\|>1$, so Proposition \ref{minl} applies.

With the notation of the theorem, let $\Dhat_1$ denote the length of the linear path $\ell_1$ in $F_0(\R^n,2)$ from $P$ to any point $(x-u,x+u)$ with $\|u\|=1$. This equals
\begin{eqnarray}&&\sqrt{\|x-u-a\|^2+\|x+u-a'\|^2}\nonumber\\
&=&\nonumber\sqrt{\|a\|^2+\|a'\|^2+2\|x\|^2+2-4x\cdot A-4h\cdot u}\\
&=&\sqrt2\sqrt{\|h\|^2+1+\|x-A\|^2-2h\cdot u}.\label{D1}\end{eqnarray}

A path  $\ell_1\sigma\ell_2$ has length $\Dhat_1+\Dhat_3+\Dhat_2$, where $\Dhat_3$ is the length of the curved path $\sigma$ described in Proposition \ref{bdrylem}, and $\Dhat_2$ is a formula similar to (\ref{D1}) for a linear path $\ell_2$ from $(y-v,y+v)$ to $Q$. Let $D_i=\Dhat_i/\sqrt2$ and $T=D_1+D_3+D_2$. If $h\cdot u=1=k\cdot v$ and $\a$ is the angle between $u$ and $v$, the formulas for $D_1$ and $D_2$ simplify nicely, and we have
\begin{equation}\label{T3}T=\sqrt{S_0^2+\|x-A\|^2}+\sqrt{\|x-y\|^2+\a^2}+\sqrt{S_1^2+\|y-B\|^2}.\end{equation}

Setting $\partial T/\partial x_i=0$ gives
\begin{equation}\label{px}\frac{x_i-A_i}{\sqrt{S_0^2+\|x-A\|^2}}+\frac{x_i-y_i}{\sqrt{\|x-y\|^2+\a^2}}=0,\end{equation}
so
$$(x_i-A_i)^2(\|x-y\|^2+\a^2)=(x_i-y_i)^2(S_0^2+\|x-A\|^2).$$
Summing over $i$ and cancelling yields $\a^2\|x-A\|^2=S_0^2\|x-y\|^2$. Now (\ref{px}) says $\a(x-A)=(y-x)S_0$. Similarly $\a(y-B)=(x-y)S_1$. Solving these equations yields (\ref{xy}), with $\b$ replaced by $\a$. This is a consequence of $\partial T/\partial x_i=0=\partial T/\partial y_i$ and (\ref{=1}).

When $x$ and $y$ are as in (\ref{xy}),
\begin{equation}\label{three}x-y=\tfrac{(A-B)\a}{\a+S_0+S_1},\quad x-A=\tfrac{(B-A)S_0}{\a+S_0+S_1},\quad y-B=\tfrac{(A-B)S_1}{\a+S_0+S_1},\end{equation}
and we obtain the dramatic simplification
$$T=\sqrt{\|A-B\|^2+(\a+S_0+S_1)^2},$$
showing clearly that we should choose $\b$ to minimize $\a$.

Note that $\b<\pi$ (so Lemma \ref{bdrylem} applies), since the only way to have $\b=\pi$ would be with the hyperplanes $h\cdot u=1$ and $k\cdot v=1$ tangent to the unit sphere, and parallel, so $\|h\|=1=\|k\|$, which we have removed from our consideration.

We must also consider changes of $T$ caused by changes in $u$ or $v$. The more general formula for $T$ at any point is
\begin{eqnarray}\nonumber T&=&\sqrt{H+1+\|x-A\|^2-2h\cdot u}+\sqrt{\|x-y\|^2+(\arccos(u\cdot v))^2}\\
&&+\sqrt{K+1+\|y-B\|^2-2k\cdot v}.\label{Te}\end{eqnarray}
At our claimed critical point, which was derived from $\frac{\partial T}{\partial x_i}=0=\frac{\partial T}{\partial y_i}$, (\ref{=1}), and minimal $\|u-v\|$, $u$ and $v$ lie in the $h$-$k$ plane by Proposition \ref{minl}(ii). Changes in $u$ or $v$ orthogonal to the $h$-$k$ plane will not affect (\ref{Te}) at this point. For Situation (c), each pair $u,v$ is determined by a choice of $w$. They lie in the $h$-$w$ plane. The analysis here applies with $h$-$k$ replaced by $h$-$w$.

 Letting $\a$ denote the angle from $v$ to $u$, and using $\sqrt{\|x-y\|^2+\a^2}$ for the  middle term of (\ref{Te}), we obtain
\begin{eqnarray}\nonumber\frac{\partial T}{\partial\a}&=&\frac{-h\cdot\frac{du}{d\a}}{\sqrt{S_0^2+\|x-A\|^2}}+\frac\a{\sqrt{\|x-y\|^2+\a^2}}\\
\label{num}&=&\frac{\a+S_0+S_1}{\sqrt{(\a+S_0+S_1)^2+\|B-A\|^2}}\biggl(\frac{-h\cdot\tfrac{du}{d\a}}{S_0}+1
\biggr),\end{eqnarray}
incorporating (\ref{three}).
We can parametrize $\R^n$ so that $v=(1,\overline0)$, $u=(\cos\a,\sin\a,\overline0)$,
and $h=(h_1,h_2,\overline0)$. We obtain $h\cdot\frac{du}{d\a}=-h_1\sin\a+h_2\cos\a$, so
\begin{eqnarray*}S_0^2-(h\cdot\tfrac{du}{d\a})^2&=&h_1^2+h_2^2-1-(h_1^2\sin^2\a+h_2^2\cos^2\a-2h_1h_2\cos\a\sin\a)\\
&=&(h_1\cos\a+h_2\sin\a)^2-1\\
&=&0.\end{eqnarray*}
Noting that $h\cdot\frac{du}{d\a}\ge0$ since $h\cdot u$ had minimal allowable value (1) at our point, we obtain $h\cdot\frac{du}{d\a}=S_0$, and so $\frac{\partial T}{\partial \a}=0$ in (\ref{num}).

We have shown that the assumption (\ref{=1}) leads to the unique critical point of $T$ described in Theorem \ref{mainthm}(b) when $h$ and $k$ are not parallel, and to any of the claimed points in Situation (c), and Lemma \ref{anglelem} says that we must have $h\cdot u\ge1$ and $k\cdot v\ge1$. If $h\cdot u=t>1$,  corresponding to a larger value of $\a$, then we can find values of $x$ and $y$ that make $\partial T/\partial x_i=0=\partial T/\partial y_i$ with formulas similar to (\ref{xy}) except that now $S_0=\sqrt{\|h\|^2+1-2t}$. An analysis similar to the above paragraph will lead to $\partial T/\partial\a>0$ for changes in the $h$-$k$ plane. So points with $h\cdot u>1$ or $k\cdot v>1$ cannot be critical points. We conclude that our critical point is a unique minimum of $T$ in Situation (b), and our claimed points are the only critical points in Situation (c).

Next we justify the next-to-last sentence of part (b) by noting that if $\b=0$, so $u=v$ and then clearly $x=y$ in (\ref{xy}), and then showing that
$$(x-u,x+u)=\tfrac{S_1}{S_0+S_1}P+\tfrac{S_0}{S_0+S_1}Q,$$
so the unique point where the lines from $P$ and $Q$ meet $\partial F_0$ is on the line connecting $P$ and $Q$.

This requires showing that
$$\tfrac{S_0B+S_1A}{S_0+S_1}\pm u=\tfrac{S_1}{S_0+S_1}(A\pm h)+\tfrac{S_0}{S_0+S_1}(B\pm k),$$
hence we need to prove
\begin{equation}\label{213}u=\frac{S_1h+S_0k}{S_0+S_1}.\end{equation}
By Proposition \ref{minl}(ii), $u$ is in the $h$-$k$ plane. We parametrize that plane so that $h=(h_1,h_2)$, $k=(k_1,k_2)$, and $u=(\cos\theta,\sin\theta)$. Since $h\cdot u=1$,
$h_2=(1-h_1\cos\theta)/\sin\theta$, so
$$S_0=\sqrt{h_1^2+\bigl(\tfrac{1-h_1\cos\theta}{\sin\theta}\bigr)^2-1}=\frac{\sqrt{\cos^2\theta-2h_1\cos\theta+h_1^2}}{|\sin\theta|}=\bigg|\frac{\cos\theta-h_1}{\sin\theta}\bigg|,$$
and similarly $S_1=|(\cos\theta-k_1)/\sin\theta|$.

Since $\theta$ is the common endpoint of (otherwise disjoint) intervals on which $h_1\cos\theta+h_2\sin\theta\ge1$ and $k_1\cos\theta+k_2\sin\theta\ge1$, the derivatives of these expressions must have opposite signs at $\theta$. The derivative of $h_1\cos\theta+h_2\sin\theta$ is
$$-h_1\sin\theta+\tfrac{1-h_1\cos\theta}{\sin\theta}\cos\theta=\tfrac{\cos\theta-h_1}{\sin\theta}.$$
Thus $\cos\theta-h_1$ and $\cos\theta-k_1$ have opposite signs. Thus
$$\frac{S_1h_1+S_0k_1}{S_0+S_1}=\frac{(\cos\theta-k_1)h_1-(\cos\theta-h_1)k_1}{-(\cos\theta-h_1)+(\cos\theta-k_1)}=\cos\theta.$$
Similarly, $(S_1h_2+S_0k_2)/(S_0+S_1)=\sin\theta$, proving (\ref{213}) and hence the next-to-last sentence of part (b) of Theorem \ref{mainthm}.

Finally, regarding the last sentence of part (b):
 If $\delta=2$, there are points $c=(1-t)a+tb$ and $c'=(1-t)a'+tb'$ such that $d(c,c')=2$. There is a unique unit ball having $cc'$ as a diameter, and the method of (b) will yield $\ell_1$ the linear path from $(a,a')$ to $(c,c')$, $\sigma$ a constant path, and $\ell_2$ the linear path from $(c,c')$ to $(b,b')$.
\end{proof}

\begin{proof}[Proof of Theorem \ref{GCthm}] Let $E_0$ denote the set of all $(P,Q)\in F_0(\R^n,2)\times F_0(\R^n,2)$ of type (c) in Theorem \ref{mainthm}, and $E_1$ it complement.
First we show that the unique geodesics at points $(P,Q)$ of $E_1$  vary continuously with $(P,Q)$, giving a geodesic motion planning rule on $E_1$.

For the type-(b) geodesics, the issue is whether $u$ and $v$, hence $\b$, vary continuously with $(P,Q)$. Small changes in $(P,Q)$ cause small changes in $h$ and $k$, and hence small changes in $u$ and $v$ of norm 1 satisfying $h\cdot u=1=k\cdot v$. If $(u,v)$ has minimal positive $\|u-v\|$ for such vectors, there will be a neighborhood of $(u,v)$ on which this is true. Alternatively, $u$ and $v$ vary continuously with $h$ and $k$ by Corollary \ref{cor1}.

The linear paths of type (a) vary continuously with the parameters. By the last two sentences of Theorem \ref{mainthm}(b), the paths in the intersection of types (a) and (b) agree, and so by the Pasting Lemma, we have a continuous choice of geodesics on $E_1$.

If $n$ is even, let $V$ be a unit-length vector field on $S^{n-1}$. In Theorem \ref{mainthm}(c), let $w=V(\frac h{\|h\|})$. This leads to a continuous choice of geodesics on $E_0$. Hence $\GC(F_0(\R^n,2))\le 1$ if $n$ is even. Since \begin{equation}F_0(\R^n,2)\approx \R^n\times\{x\in\R^n:\|x|\ge2\}\simeq S^{n-1}\label{he}\end{equation} and $\TC$ is a homotopy invariant, we obtain, for $n$ even,
$$\GC(F_0(\R^n,2))\ge\TC(F_0(\R^n,2))=\TC(S^{n-1})=1\ge\GC(F_0(\R^n,2)).$$
Hence we have equality.

If $n$ is odd, let $V$ be a unit-length vector field on $S^{n-1}-\{(1,\overline0)\}$. Then $w=V(\frac h{\|h\|})$ in \ref{mainthm}(c) leads to a continuous choice of geodesics on $E_0-Z$, where $Z$ is the set of $((a,a'),(b,b'))\in E_0$ such that $a'-a$ and $b'-b$ are scalar multiples of $(1,\overline0)$. On $Z$, you could use the geodesics obtained from using $w=(0,1,\overline 0)$ in \ref{mainthm}(c). Thus $\GC(F_0(\R^n,2))\le 2$, and since $\TC(S^{n-1})=2$, we have equality as in the previous paragraph.

\end{proof}
\section{Examples when $n=2$}\label{2sec}
We illustrate two examples of geodesics when $n=2$.

 Let $P=(P_1,P_2)=((-6,4),(6,8))$ and $Q=(Q_1,Q_2)=((8,-6),(2,-10))$. We have $h=(6,2)$ and $k=(-3,-2)$.
From (\ref{xy}), we obtain $x=(3.1596,-2.8468)$ and $y=(3.2474,-3.0927)$. From Proposition \ref{minl}, we obtain $u=(.4622,-.8867)$ and $v=(.3022,-.9533)$, and from Corollary \ref{cor1}, $\b=.1736$.  In this example, our path has length 25.2455, whereas the straight line path from $P$ to $Q$ (which is not in $F_0(\R^2,2)$) has length 25.2190.
Although we cannot quite draw the short middle part of the paths, in Figure \ref{fig3} we picture the paths in this example.

\bigskip
\begin{minipage}{6in}
\begin{fig}\label{fig3}

{\bf Example of geodesic.}

\begin{center}

\begin{\tz}[scale=.35]
\draw (3.16,-2.847) circle [radius=1];
\draw (3.247,-3.093) circle [radius=1];
\draw (-6,0) -- (8,0);
\draw (0,-10) -- (0,8);
\draw (6,8) -- (3.622,-3.733);
\draw [fill=red] [thick] (-6,4) -- (2.7,-1.96);
\draw (2,-10) -- (3.55,-4.046);
\draw [fill=red] [thick] (8,-6) -- (2.945,-2.14);
\node at (6,8) {$\ssize\bullet$};
\node at (-6,4) {$\ssize\bullet$};
\node at (2,-10) {$\ssize\bullet$};
\node at (8,-6) {$\ssize\bullet$};
\node at (7,8) {$P_2$};
\node at (3,-10) {$Q_2$};
\node at (-5,4) {$P_1$};
\node at (7,-6) {$Q_1$};
\end{\tz}
\end{center}
\end{fig}
\end{minipage}

Now we change the $8$ in $P_2$ of the above example to $12$, so that $\vec{P_1P_2}$ and $\vec{Q_1Q_2}$ are parallel in opposite directions; i.e., $h=(6,4)$ and $k=(-3,-2)$.

The two equal paths are depicted in Figure \ref{fig4}. We have $w=\pm(2,-3)/\sqrt{13}$ in part (c) of Theorem \ref{mainthm}. Also, $x=(3.2385,-2.3633)$ and $y=(3.4291,-2.9730)$ in (\ref{xy}) and $\b=.4202$ in Corollary \ref{cor1}. The length of each of these paths in $F_0(\R^n,2)$ is 28.375, compared with 28.213 for the straight-line path from $P$ to $Q$, which is not in $F_0(\R^n,2)$.

\bigskip
\begin{minipage}{6in}
\begin{fig}\label{fig4}

{\bf Example of two geodesics.}

\begin{center}

\begin{\tz}[scale=.35]
\draw (3.2385,-2.3633) circle [radius=1];
\draw (3.4291,-2.9730) circle [radius=1];
\draw (-6,0) -- (8,0);
\draw (0,-10) -- (0,8);
\draw [fill=red] [thick] (6,12) -- (3.9033,-3.1104);
\draw (6,12) -- (2.8046,-1.4624);
\draw [fill=red] [thick] (-6,4) -- (2.5738,-1.6162);
\draw (-6,4) -- (3.6725,-3.2643);
\draw [fill=red] [thick] (2,-10) -- (3.7312,-3.9263);
\draw (2,-10) -- (2.6654,-2.3275);
\draw [fill=red] [thick] (8,-6) -- (3.1269,-2.0198);
\draw (8,-6) -- (4.1928,-3.6186);
\node at (6,12) {$\ssize\bullet$};
\node at (-6,4) {$\ssize\bullet$};
\node at (2,-10) {$\ssize\bullet$};
\node at (8,-6) {$\ssize\bullet$};
\node at (7,12) {$P_2$};
\node at (3,-10) {$Q_2$};
\node at (-5,4) {$P_1$};
\node at (7,-6) {$Q_1$};
\end{\tz}
\end{center}
\end{fig}
\end{minipage}

\section{Unordered configuration space}\label{unorsec}
As we stated in Theorem \ref{Cthm}, the unordered configuration space $C(\R^n,2)$ has a natural, Euclidean-related metric, and it is geodesically complete.

\begin{proof}[Proof of Theorem \ref{Cthm}] We show that $d_U$ satisfies the triangle inequality. Without loss of generality, assume
$$d((a,a'),(b,b'))\le d((a,a'),(b',b))\quad\text{and}\quad d((b,b'),(c,c'))\le d((b,b'),(c',c)).$$
Then
\begin{eqnarray*}d_U(\{a,a'\},\{c,c'\})&\le&d((a,a'),(c,c'))\\
&\le&d((a,a'),(b,b'))+d((b,b'),(c,c'))\\
&=&d(\{a,a'\},\{b,b'\})+d(\{b,b'\},\{c,c'\}).\end{eqnarray*}
The quotient topology on $C(\R^n,2)$ comes from $\R^n\times(\R^n-\{0\})/\sim$ under $\{a,a'\}\mapsto\bigl(\frac{a+a'}2,\bigl[\frac{a-a'}2\bigr]\bigr)$. Our metric $d_U$ corresponds to the metric $d([x],[x'])=\min(d(x,x'),d(x,-x'))$ on $(\R^n-\{0\})/\sim$, which gives the quotient topology.

If $d_U(\{a,a'\},\{b,b'\})=d((a,a'),(b,b'))$, then the linear path $(1-t)(a,a')+t(b,b')$ lies in $F(\R^n,2)$, so its equivalence class is in $C(\R^n,2)$. [\![As observed at the beginning of the proof of Theorem \ref{mainthm}, the only thing that would prevent the path from being in $F(\R^n,2)$ is if $b'-b$ is a negative multiple of $a'-a$. If this is the case, then
\begin{eqnarray}&&d((a,a'),(b,b'))^2-d((a,a'),(b',b))^2\nonumber\\
&=&\|a-b\|^2+\|a'-b'\|^2-\|a-b'\|^2-\|a'-b\|^2\nonumber\\
&=&-2a\cdot b-2a'\cdot b'+2a\cdot b'+2a'\cdot b\nonumber\\
&=&2(a'-a)\cdot(b-b')\label{bb}\\
&>&0,\nonumber\end{eqnarray} contradicting the assumption that $d_U(\{a,a'\},\{b,b'\})=d((a,a'),(b,b'))$.]\!]
\end{proof}

\begin{rmk} {\rm The analogue of Theorem \ref{Cthm} is valid for $C(\R^n,k)$ for any $n$ and $k$.}\end{rmk}

\begin{prop} For $a$, $a'$, $b$, and $b'$ in $\R^n$, $d((a,a'),(b,b'))=d((a,a'),(b',b))$ iff $(a'-a)\cdot (b'-b)=0$. Let $$E_0=\{(\{a,a'\},\{b,b'\})\in C(\R^n,2)\times C(\R^n,2):(a'-a)\cdot (b'-b)\ne0\}.$$ There is a continuous geodesic motion-planning rule on $E_0$.\label{E0}\end{prop}
\begin{proof} The first part follows as in (\ref{bb}).
At each point of $E_0$, there is a unique choice of linear geodesic whose length equals $d_U(\{a,a'\},\{b,b'\})$, varying continuously with the point of $E_0$.
\end{proof}

Now we can prove Theorem \ref{GCCthm} about $\GC(C(\R^n,2))$.
\begin{proof}[Proof of Theorem \ref{GCCthm}]
Let $E_1=C(\R^n,2)\times C(\R^n,2)-E_0$, with $E_0$ as in Proposition \ref{E0}. We need to describe subsets of $E_1$ on which we can make a continuous choice of whether to go from $(a,a')$ to $(b,b')$, or  to $(b',b)$. Let $t_n=\TC(RP^{n-1})$. One motion-planning rule for $RP^{n-1}$ is on the domain $D_0$ consisting of all pairs $(\ell_0,\ell_1)$ such that $\ell_0\cdot\ell_1\ne0$. On $D_0$, rotate from $\ell_0$ to $\ell_1$ in their plane through the smaller arc ($<\pi/2$).
Suppose $D_i$ is one of the other $t_n$ subsets of $RP^{n-1}\times RP^{n-1}$ on which there is a motion-planning rule $s_i$. For $(\ell_0,\ell_1)\in D_i$, $s_i$ associates to an orientation (i.e., direction) on $\ell_0$ an orientation on $\ell_1$ by using the path from $\ell_0$ to $\ell_1$ specified by $s_i$. This association of orientations is continuous on $D_i$.

Let $E_i\subset E_1$ denote all $(\{a,a'\},\{b,b'\})$ such that, if $\ell_0$ is the line through $a$ and $a'$, translated to pass through $\overline 0$, and $\ell_1$ the translated line passing through $b$ and $b'$, then $(\ell_0,\ell_1)\in D_i$.
If $a>a'$ under the orientation of $\ell_0$, choose  the linear geodesic from $(a,a')$ to $(b,b')$, where $b>b'$ under the associated orientation of $\ell_1$. This is a continuous choice on $E_i$.

Thus we have partitioned $C(\R^n,2)\times C(\R^n,2)$ into $t_n+1$ subsets on which we have geodesic motion planning rules, so $\GC(C(\R^n,2))\le\TC(RP^{n-1})$. Since, from (\ref{he}), $C(\R^n,2)$ has the homotopy type of $RP^{n-1}$, we have the following string of inequalities, which imply the claimed equality.
$$\GC(C(\R^n,2))\ge\TC(C(\R^n,2))=\TC(RP^{n-1})\ge\GC(C(\R^n,2)).$$
\end{proof}
Note that our argument did not use that the motion-planning rules of \cite{FTY} can be chosen to be geodesics.

\section{A different metric}\label{sec4}
There is an obvious homeomorphism $F(\R^n,2)\to \R^n\times S^{n-1}\times \R^+$, where $\R^+=(0,\infty)$, given by
$$(a,a')\mapsto \biggl(\frac{a'+a}2, \frac{a'-a}{\|a'-a\|}, \frac{\|a'-a\|}2\biggr).$$
In the notation of Theorem \ref{mainthm}, with $\hh=h/\|h\|$, it is
$(a,a')\mapsto(A,\hh,\|h\|)$.
Here $A$ is the midpoint, and $h$ the directed segment from the midpoint to the second point.
The inverse sends $(A,u,r)$ back to $(A-ru,A+ru)$.

We use the Euclidean metric on $\R^n$ and $\R^+$, and arclength metric $d_S$ on $S^{n-1}$, and the product metric on their product to obtain a metric $d'$ on $F(\R^n,2)$ which is geodesically complete. The formula, with $B$ and $k$ also as in Theorem \ref{mainthm}, is
$$d'((a,a'),(b,b'))=\sqrt{\|B-A\|^2+d_S(\hh,\kk)^2+(\|k\|-\|h\|)^2}.$$
The unique geodesic from $(a,a')$ to $(b,b')$, if $d_S(\hh,\kk)<\pi$, is
$$t\mapsto ((1-t)A+tB-((1-t)\|h\|+t\|k\|)u(t),(1-t)A+tB+((1-t)\|h\|+t\|k\|)u(t)),$$
where, similarly to Proposition \ref{bdrylem}, with $\a=d_S(\hh,\kk)$, $$u(t)=\frac{\sin((1-t)\a)\hh+\sin(t\a)\kk}{\sin\a}.$$
If $d_S(\hh,\kk)=\pi$, we use vector fields on $S^{n-1}$ or $S^{n-1}-\{x_0\}$ to choose geodesics, and obtain an analogue of Theorem \ref{GCthm} for $F(\R^n,2)$ in this metric.
Figure \ref{paths} shows the path obtained in this way between the points that we used in the first example of Section \ref{2sec}.

\bigskip
\begin{minipage}{6.5in}
\begin{fig}
\label{paths}
{\bf Geodesic in $F(\R^2,2)$ using ``different'' metric}
\begin{center}
\includegraphics[width=4.5in,height=4.5in]{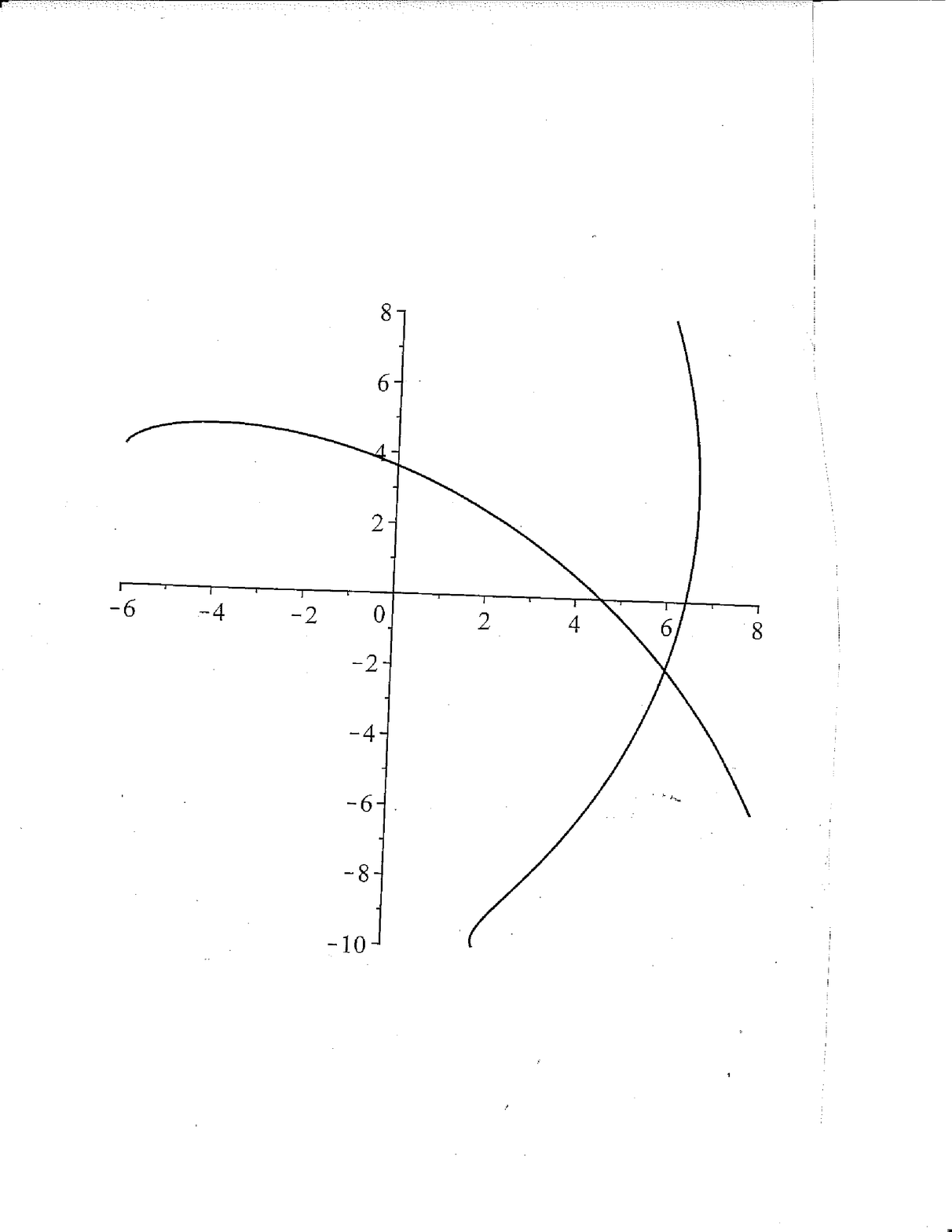}
\end{center}
\end{fig}
\end{minipage}
\bigskip

There is an analogous metric on $C(\R^n,2)$, but we will not discuss it here because the Euclidean-related metric considered earlier was already geodesically complete and highly satisfactory.

\def\line{\rule{.6in}{.6pt}}

\end{document}